\documentclass[11pt]{article}
\usepackage{graphicx}
\usepackage{tikz}
\usepackage{amssymb}
\usepackage{amsmath,amsthm}
\usepackage{verbatim}
\usepackage{amsfonts}
\usepackage[T1]{fontenc}
\usepackage{multicol}
\usepackage{color}

\newtheorem{thm}{Theorem}
\newtheorem{prop}{Proposition}[section]

\newtheorem{defn}[thm]{Definition}
\newtheorem{obs}[thm]{Observation}
\newtheorem{lem}[thm]{Lemma}

\newcommand{\E}{\mathcal{E}}

\newcommand{\G}{\mathcal{G}}
\newcommand{\Hy}{\mathcal{H}}

\newcommand{\N}{\mathcal{N}}

\newcommand{\bP}{\mathbb{P}}
\newcommand{\bE}{\mathbb{E}}

\newcommand{\bN}{\mathbb{N}}
\newcommand{\bfa}{\mathbf{a}}
\newcommand{\bfb}{\mathbf{b}}
\newcommand{\bfg}{\mathbf{g}}
\newcommand{\bfx}{\mathbf{x}}
\newcommand{\bfy}{\mathbf{y}}
\newcommand{\bfzero}{\mathbf{0}}

\newcommand{\X}{\mathcal{X}}
\newcommand{\Y}{\mathcal{Y}}
\newcommand{\ep}{\varepsilon}
\newcommand{\fD}{f^{(D)}}
\newcommand{\Z}{\mathbb{Z}}
\newcommand{\no}{\noindent}

\title{The Weighted Davenport constant of a group and a related extremal problem}
\author{Niranjan Balachandran\footnote{Department of Mathematics, 
Indian Institute of Technology Bombay, Mumbai, India. 
email: niranj (at) math.iitb.ac.in}\ \  and 
Eshita Mazumdar\footnote{Department of Mathematics, Indian 
Institute of Technology Bombay, and Center for Combinatorics, 
Nankai University, Tianjin, China. email: eshitamazumdar (at) yahoo.com} }

\begin{document}
\maketitle
\begin{abstract} 
For a finite abelian group $G$ written additively, and a non-empty subset $A\subset [1,\exp(G)-1]$ the weighted Davenport Constant of $G$ 
with respect to the set $A$, denoted $D_A(G)$, is the least positive integer $k$ for which the following holds: Given an arbitrary $G$-sequence 
$(x_1,\ldots,x_k)$, there exists a non-empty subsequence  $(x_{i_1},\ldots,x_{i_t})$ along with $a_{j}\in A$ such that $\sum_{j=1}^t a_jx_{i_j}=0$. In this paper, we pose and study a natural new extremal problem that arises from the study of $D_A(G)$:  For an integer $k\ge 2$, determine $\fD_G(k):=\min\{|A|: D_A(G)\le k\}$ (if the problem posed makes sense). It turns out that for $k$ `not-too-small', this is a well-posed problem and one of the most interesting cases occurs for $G=\Z_p$, the cyclic group of prime order, for which we obtain near optimal bounds for all $k$ (for sufficiently large primes $p$), and asymptotically tight (up to constants) bounds for $k=2,4$.
\end{abstract}


2010 AMS Classification Code:  11B50, 11B75, 05D40.

\section{Introduction} By $[n]$ we shall mean the set $\{1,\ldots,n\}$ and by $[a,b]$ the set $\{a,a+1,\ldots,b\}$ for integers $a\le b$. Throughout this paper, we shall use the Landau asymptotic notation: For functions $f,g$, we write $f(n)=O(g(n))$ if there exists an absolute constant $C>0$ and an integer $n_0$ such that for all $n\ge n_0, |f(n)|\le C|g(n)|$. We write $f=\Omega(g)$ if $g=O(f)$, and we write $f=\Theta(g)$ if $f=O(g)$ and $f=\Omega(g)$. By $f\ll g$ we mean $\frac{f(n)}{g(n)}\to 0$ as $n\to\infty$. We shall also use some of the standard notation from additive combinatorics: For sets $A,B$ subsets of the cyclic group $\Z_n$,  $A+B:=\{a+b:a\in A,b\in B\}$ and $\alpha A=\{\alpha a:a\in A\}$. 

Let $G$ be a finite abelian group written additively. By a 
$G$-sequence of length $k$, we shall mean a sequence 
$(x_1,\ldots,x_k)$ with $x_i\in G$ for each $i$. By a 
\textit{zero-sum} $G$-sequence (or simply, zero-sum sequence) we 
shall mean a $G$-sequence $(x_1,\ldots,x_k)$ such that $\sum_i x_i = 0$, where $0$ is the identity element of $G$. The \textit{Davenport Constant} $D(G)$,  introduced by Rogers \cite{Rog}, is defined as the smallest $k$ such that every $G$-sequence of length $k$ contains a non-trivial zero-sum subsequence. As it turns out, the Davenport constant is an important invariant of the ideal class group of the ring of integers of an algebraic number field (see \cite{halter} for more details). 

A weighted version of the Davenport constant which first appeared 
in a paper by Adhikari \textit{et al} (\cite{ACFKP}), and was later generalized 
by Adhikari and Chen (\cite{AC}), goes as follows. Suppose $G$ is a finite 
abelian group, and let $A\subset\Z\setminus\{0\}$ be a non-empty subset of 
the integers. The weighted Davenport Constant of $G$ 
\textit{with respect to the set} $A$ is the least positive integer $k$ for 
which the following holds: Given an arbitrary $G$-sequence 
$(x_1,\ldots,x_k)$, there exists a non-empty subsequence 
$(x_{i_1},\ldots,x_{i_t})$ along with $a_{j}\in A$ such that 
$\displaystyle\sum_{j=1}^t a_jx_{i_j}=0$. Here we adopt the convention 
that $ax:=\overbrace{x+\cdots+x}^{a\textrm{\ times}}$ for $a$ positive, and $ax:= (-a)(-x)$ for $a$ negative. It is clear that if $G$ has exponent $n$, then one may restrict $A$ to be a subset of $[1,n-1]$. 

As one might expect, the Davenport constant is best understood 
when $G$ is a finite cyclic group. Here are some well-known 
results:
\begin{enumerate}
\item $D_{\pm}(\Z_n)=\lfloor\log_2 n\rfloor +1$. Here our 
notation is a shorthand to denote that the set $A=\{-1,1\}$.
(\cite{ACFKP})
\item $D_A(\Z_n)=2$ if $A=\Z_n\setminus\{0\}$.(\cite{ACFKP})
\item $D_A(\Z_n) = a+1 $ if $A=\Z_n^*$ and $a=\sum_{i=1}^{k}a_i$ 
 for $n = p_1^{a_1}p_2^{a_2}\ldots p_k^{a_k}$. (\cite{Gri})
\item $D_A(\Z_n)=\lceil\frac{n}{r}\rceil $ if $A=\{1,\ldots,r\}$ 
for some $1\le r\le n-1$. (\cite{Adh et al}, \cite{AR})
\end{enumerate}
For other results, see the papers \cite{AR}, \cite{Adetal}.

The focal point of this paper stems from a natural extremal problem in light of the known results  on the Davenport constant of a group. Suppose $G$ is a finite abelian group of exponent $n$, and let $k\ge 2$ be an integer. Define
\begin{eqnarray*}\label{fdG_def} \fD_G(k)&:=&\min\{|A|: \emptyset \neq A\subseteq[1,n-1]\textrm{\ satisfies\ }D_A(G)\le k\},\\ 
                                         &:=&\infty\textrm{\ 
                                         if\ there\ is\ no\ 
                                         such\ }A.\end{eqnarray*}
                                         Here is a natural extremal problem: Given a finite abelian group $G$, determine $\fD_G(k)$ for $k\in\bN$.
                                         
 It is important to note that if $k$ is `too small' relative to the group, then the parameter as defined above is in fact infinite. For instance, consider the group $G=\Z_p^r$ for $p$ prime, and the sequence $\bfx:=(e_1,\ldots,e_r)$ where $e_i=(0,\ldots,0,1,0\ldots,0)$ has a $1$  in the $i^{th}$ coordinate, for $1\le i\le r$. Then for any subset  $A\subset\Z_p^*$ and any sequence $(a_1,a_2,\ldots,a_r)$ with $a_i\in A$ the element $\displaystyle \sum_{i=1}^r a_ie_i=0$ implies $a_i=0$ for each $i$, which implies that $\fD_G(k)=\infty$ for $k\le r$. However, for $k>r$ we do have $\fD_G(k)<\infty$. As it turns out, a consequence of one of our theorems implies that for every group, if $k$ is not too small (this will be more clear when we see the statement of the theorem), then $\fD_G(k)<\infty$, so this is indeed a non-trivial parameter.  For $G=\Z_n$, we shall write $\fD(n,k):= \fD_G(k)$ for convenience. 
 
As it turns out, the  the nature of this extremal problem of determining $\fD_G(k)$ is most interesting for the case when $G$ is a cyclic group of prime order, and in that case, we establish the following bounds.

\begin{thm}\label{fD_prime}
Let $k\in\bN$, $k\ge 2$. There exists an integer $p_0(k)$ and an absolute constant $C=C(k)>0$ such that for all prime $p>p_0(k)$
$$ p^{1/k}-1\le \fD(p,k)\le C(p\log p)^{1/k}.$$\end{thm}

As some of our preliminary results will illustrate, this also determines asymptotically (up to a logarithmic factor) $\fD_G(k)$ for $G=\Z_p^r$ and $G=\Z_{p^r}$ and these in turn determine $\fD_G(k)$ in most cases since this theorem determines $\fD(p,k)$ up to a logarithmic factor for all integers $k\ge 2$. 

In a couple of special cases, viz., $k=2, 4$, we are able to obtain an asymptotically sharper upper bound which is tight upto constant factors.  In fact, for the case $k=2$, the extremal problem even achieves tight bounds in certain special cases. As we shall see, the case of $k=2$ is related to an older problem of determining minimal sized difference bases of $[p]=\{1,\ldots,p\}$.  

\begin{thm}\label{2and4} Let $p$ be an odd prime.
\begin{enumerate}
\item If $p=q^2+q+1$ for some prime $q$ then $\fD(p,2)=\lceil\sqrt{p-1}\rceil$.
\item $\fD(p,2)\le 2\sqrt{p}-1$.
\item $\fD(p,4)\le Cp^{1/4}$ for some absolute constant $C>0$.\end{enumerate}\end{thm}

In particular, this theorem establishes that $\fD(p,k)$ is of the order of $p^{1/k}$ for $k=2,4$ upto a constant factor. As for the tightness result, there is an old conjecture of Hardy-Littlewood that there are infinitely many prime pairs $(p,q)$ such that $p=q^2+q+1$.  


The rest of the paper is organized as follows. In the next section, we consider the extremal problem more formally, and prove a few lemmas that substantiate our claim that the problem is most interesting in the case when the group is cyclic of prime order. In the next section, we prove theorems \ref{fD_prime} and \ref{2and4}. We conclude the paper with some remarks and open questions.

Before we end this section, we set up some notation and terminology.  For a sequence $\bfx=(x_1,\ldots,x_m)$ and an integer $\lambda$, by $\lambda\cdot\bfx$, we shall mean the sequence $(\lambda x_1,\ldots,\lambda x_m)$. For a subset $I\subset [m]$ of the set of indices, we shall denote by $\bfx_I$ the sum $\displaystyle\sum_{i\in I} x_i$, and for sequences $\bfx=(x_1,\ldots,x_m), \bfy=(y_1,\ldots,y_m)$ of the same length, we shall denote by $\langle \bfx,\bfy\rangle_I$ the sum $\displaystyle\sum_{i\in I} x_iy_i$ where as before, $I\subset[m]$. In case $I=[m]$ then we shall drop the subscript and simply write $\langle\bfx,\bfy\rangle$ to denote $\langle\bfx,\bfy\rangle_{[m]}$.

\section{The extremal problem of $\fD_G(k)$: Preliminaries}
Let $k\ge 2$ be an integer. We start by recalling the definition of $\fD_G(k)$: 
\begin{eqnarray}\label{fdG_def} \fD_G(k):=\min\{|A|: A\subseteq[1,n-1]\textrm{\ satisfies\ }D_A(G)\le k\}.\end{eqnarray}
We shall (as mentioned in the introduction) write $\fD(n,k)$ to denote $\fD_G(k)$ for $G=\Z_n$.

\begin{prop}\label{HxK} Let $G=H_1\times\cdots \times H_r$ be the product of $p_i$-groups $H_i$ with $p_1<p_2<\ldots<p_r$. Then for all $k$, $\fD_G(k)\le\min\{\fD_{H_i}(k):1\le i\le r\}$. In particular, 
if $G= \Z_p\times H$ is a finite abelian group with $p\nmid |H|,$ then for any integer $k$, $ f_G^{(D)}(n,k)\leq f^{(D)}(p,k).$ \end{prop}

\begin{proof} Write $n_i=\exp(H_i)$, and let $n:=\exp(G)=\prod_{i=1}^r n_i$. 

Suppose $\fD_{H_i}(k) = l$, and let $A_i:=\{a_1,\ldots,a_{l}\}\subset [1,n_i-1]$ be such that $D_{A_i}(H_i)=k$. Then consider the set $$A=\{(n/n_i)a_1,\ldots,(n/n_i)a_l\}\subset [1,n-1].$$ We claim that for any $G$-sequence $\bfg=(g_1,\ldots,g_k)$  of length $k$ there exists $\bfa=(a_1,\ldots,a_k) \in (A\cup \{0\})^k \setminus\{\bfzero_k\},$ for which $\langle \bfg,\bfa\rangle=0$ in $G$.

 Write $g_i=(x_i,y_i)$ for $i\in [k]$, where $x_i \in H_i$ and $y_i \in H:=\prod_{j\ne i} H_j$.  Since $\exp(H)= n/n_i$ we have $(n/n_i)y_i = 0$ for all $i$. Furthermore, by the assumption, there exists  $\bfa=(a_{1},\ldots, a_{k}) \in (A\cup \{0\})^k \setminus\{\bfzero_k\}$ such that 
 $\langle \bfx,\bfa\rangle=0$, where $\bfx=(x_1,\ldots,x_k)$. Consequently,  
 $$\sum_{j=1}^{k} \big(n/n_i)  a_{j}\big) (x_j,y_j) = \left(\sum_{j=1}^{k} (n/n_i) a_{j}x_j,\sum_{j=1}^{k}(n/n_i)a_{j}y_j\right)= (0,0)$$ and that completes the proof.
 
 The second part of the statement is an immediate consequence of the first part.
 \end{proof}

The next proposition compares groups $G,G'$ with the same exponent.
\begin{prop}\label{fD_ZpxZp}
Let $k\ge 2$. Suppose $G$ and $H$ are finite abelian groups with $H=G\times G'$ and $\exp(G)=\exp(H)$. Then 
$$ \fD_G(k)\le\fD_{H}(k).$$ In particular, if $G_n=(\Z_p)^n$ and we write $f_n:=\fD_G(k)$, then the sequence $\{f_n\}_{n\ge 1}$ is increasing. 
 \end{prop}
\begin{proof} Let $f_{H}^{(D)}(k)=\ell$. Let $A\subset[1,\exp(G)-1]$ with $|A|=\ell$ such that for all $H$-sequences $\bfx$ of length $k$, there exists $\bfa\in(A\cup\{0\})^k\setminus\{\bfzero_k\}$ such that $\langle \bfx,\bfa\rangle=0$. Let $\bfy=(y_1,\ldots,y_k)$ be a $G$-sequence of length $k$, and consider the  sequence $\bfx=(x_1,\ldots,x_k)$, where $x_i= (y_i,0)\in H$. Let 
$\bfa=(a_1,\ldots,a_k) \in (A\cup \{0\})^k \setminus\{\bfzero_k\}$ be such that $\langle \bfx,\bfa\rangle=0$ in $H$. But since $\exp(G)=\exp(H)$, this implies that $\langle  \bfy,\bfa\rangle=0$
in $G$ as well. This completes the proof. 

The second part is again a straightforward consequence of the first statement.
\end{proof}

 The next theorem contrasts $\fD(p,k)$ with $\fD_G(k)$ for $G=\Z_{p^m}$.

\begin{thm}\label{bndsZp^m} Let $p$ be a prime and $m\geq 1, k\geq 2$ be positive integers., Then for $G= \Z_{p^m},$
$$p^{1/k}-1\le f_G^{(D)}(k)= f^{(D)}(p,k).$$
\end{thm}
\begin{proof} We first prove that $\fD_G(k)=\fD(p,k)$. Let $f^{(D)}(p,k)= l.$ Then there exists $A= \{ a_1,\cdots, a_l\}\subset \Z_p^*$ of size $l$ such that
for every $\Z_p$-sequence $\bfx:=(x_1,\ldots,x_k)$ of length $k$, there exists $\bfa=(a_{i_1},\cdots,a_{i_k}) \in (A\cup \{0\})^k \setminus\{\bfzero_k\}$ such that $\langle \bfa,\bfx\rangle =0$. Set $A'= p^{m-1}\cdot A \subset [p^m-1]$. Consider a $\Z_{p^m}$-sequence $\bfx':=(x_1',\ldots,x_k')$ of length $k$ and let $x_i$ denote the projection of $x_i'$ on $\Z_p.$
By our assumption, there exist $(a_{i_1},\ldots,a_{i_k}) \in (A\cup \{0\})^k \setminus \{\bfzero_k\},$ such that $\langle \bfa,\bfx' \rangle =0$. Consequently, 
$\langle p^{m-1}\cdot\bfa,\bfx'\rangle =0$ in $G$. This establishes that $\fD_G(k)\le \fD(p,k)$.

 We shall now prove the reverse inequality, i.e., $\fD_G(k)\ge\fD(p,k)$ by means of contradiction.  Clearly, $m\ge 2$, or there is nothing to prove. Suppose if possible that there exists $A\subset [1,p^{m}-1]$ with $|A|<\fD(p,k)$ such that for every $G$-sequence $\bfx=(x_1,\ldots,x_k)$ of length $k$ there exists $\bfa\in (A\cup\{0\})^k\setminus\{\bfzero_k\}$ such that $\langle\bfa,\bfx\rangle=0$ in $G$. Let us write 
 $$A=A_0\cup (p\cdot A_1)\cup\cdots\cup \left(p^{m-1}\cdot A_{m-1}\right)$$ with $A_i\subset\Z_{p^{m-i}}^*$ for each $0\le i\le m-1$. Let $A_i':= A_i\pmod p$, and let $B=\cup_{i=0}^{m-1} A_i'$. Clearly, $B\subset [1,p-1]$ and $|B|\le |A|<\fD(p,k)$. 
 
 Let $\bfy=(y_1\ldots,y_k)$ be a $\Z_p$-sequence of length $k$. Viewing this as a $G$-sequence, and using the property of $A$, it follows that there exists $\bfa\in (A\cup\{0\})^k\setminus\{\bfzero_k\}$ such that
 \begin{eqnarray*}\label{fDk} \sum_{A_0} a_i y_i + p\cdot\sum_{A_1} a_iy_i +\cdots+p^{m-1}\sum_{A_{m-1}} a_iy_i\equiv 0\pmod{p^m}.\end{eqnarray*}
Note that in this notation if all the $a_i$ listed in a particular summand are zero, then we treat that summand as empty. 

Now, if the first summand is non-empty, then in particular, we 
must have $\sum_{A_0'} a_i' y_i=0$ in $\Z_p$; here by $a_i'$ we 
mean the corresponding projection of $a_i$ into the set $A_0'$. In 
general, let the first non-empty summand in (\ref{fDk}) is 
$p^j\sum_{A_j} a_i y_i$, then note that considering (\ref{fDk}) 
modulo $p^j$ it follows that $\sum_{A_j'} a_i'y_i=0$ in $\Z_p$. 
Since at least one summand is non-empty, this yields a non-empty 
subsequence of $\bfy$ that admits a $B$-weighted zero-sum 
subsequence in $\Z_p$, contradicting that $|B|<\fD(p,k)$. This 
completes the other inequality, and thereby establishes 
$\fD_G(k)=\fD(p,k)$ for $G=\Z_{p^m}$.

 We finally prove that $\fD(p,k)\ge p^{1/k}-1$. Consider a 
 bipartite graph $\G=(V,E)$ with $V=\X\cup \Y$, where $\X$ 
 consists of all $k$-tuples 
 $\bfa\in(A\cup\{0\})^k\setminus\{\bfzero_k\}$ and $\Y$ consists 
 of all $k$-tuples $\bfx= (x_1,x_2,\ldots,x_k)$ where all the 
 $x_i$'s are non-zero elements in $\Z_p$, and $\bfa,\bfx$ are 
 adjacent in $\G$ if and only if $\langle\bfa,\bfx\rangle=0$ in 
 $\Z_p$. By the hypothesis on $A$, it follows that every vertex of 
 $\Y$ is adjacent to at least one vertex of $\X$, so $\G$ has at 
 least $(p-1)^k$ edges. On the other hand, fix $\bfa\in \X$, and 
 assume without loss of generality that $a_1\ne 0$. Then for any 
 possible choices for $x_2,\ldots,x_k\in\Z_p^*$, the equation 
 $a_1x_1=-(a_2x_2+\cdots+a_kx_k)$ admits a unique solution for 
 $x_k\in\Z_p$, so that the vertex $\bfa\in \X$ has degree at most 
 $(p-1)^{k-1}$. Hence $$|\X|(p-1)^{k-1}\geq |E|\geq (p-1)^k,$$
and since $|\X|=(|A|+1)^k-1$, it follows that $$|A|=f^{(D)}(p,k) \geq p^{1/k}-1$$ and that completes the proof.





\end{proof}
 The last part of the proof of theorem \ref{bndsZp^m} in fact can be modified \textit{mutatis mutandis} to also show that $\fD_G(k)\ge |G|^{1/k} -1$ holds for $G=\Z_p^s$. To reiterate a point we mentioned in the introduction, $\fD_G(k)=\infty$ for $k\le s$. In light of the remark above, it is somewhat natural that we turn our attention to the case  $G=\Z_n^s$. The following proposition shows that for $k>s+1$ and $p$ reasonably large, the parameter $\fD_G(k)<\infty$.

\begin{prop}\label{pm1-pmr}
Let $G= \Z_{n_1}\times \cdots \times \Z_{n_s}$, where $1<n_1\mid\cdots\mid n_s$. 
Let $1\le r<(n-1)/2$, and let $A= \{\pm 1,\pm 2,\cdots,\pm r\}$. Then 
\begin{eqnarray*} 1+\sum_{i=1}^s{\left\lceil\log_{r+1}n_i\right\rceil}\geq D_A(G)&\geq& 1+\sum_{i=1}^s{\left\lfloor\log_{r+1}n_i\right\rfloor} \textrm{\ for\ }s\ge 2\\
                                D_A(\Z_n)&=&{\left\lfloor\log_{r+1}n\right\rfloor}+1.\end{eqnarray*}
 Consequently, $\fD(n,k)\le 2(n^{1/(k-1)}-1)$, and $\fD_G(k)\le 2(|G|^{1/(k-s-1)}-1)$ for $s>1$.
\end{prop}
\begin{proof} Consider the following sequence of elements of $G$:
$$\bfa:= (1,\ldots,0),((r+1),\ldots,0)\cdots,((r+1)^{t_1},\ldots,0),
\ldots,(0,\ldots,1),(0,\ldots,(r+1)),\ldots,(0,\ldots,(r+1)^{t_s}),$$
where $t_i$ is defined by $(r+1)^{t_i+1}\leq n_i < (r+1)^{t_i+2}$
for $1\leq i\leq s$. If $t$ is such that $(r+1)^{t+1}\leq n < (r+1)^{t+2}$
then by the choice of $t$, all integers of the form $a_0+a_1(r+1)+\cdots+a_t(r+1)^t$ 
where $a_i\in[r]$ are strictly less than $n$, and are 
distinct in $\Z_n$, it follows that no element of the form 
$a_0+a_1(r+1)+\cdots+a_t(r+1)^t$ where $a_i\in A$ equals 
zero in $\Z_n$. In particular, it follows that the sequence 
$\bfa$ admits no non-trivial zero sum subsequence. 
Furthermore, since $\bfa$ has $\sum_{i=1}^s(t_i+1) = 
\sum_{i=1}^s{\left\lfloor\log_{r+1}n\right\rfloor}$ elements, 
we have $D_A(\Z_n) \geq \sum_{i=1}{\left\lfloor\log_{r+1}n\right\rfloor}+1$.

To prove the upper bound, consider a sequence 
$\bfx=(x_1,\ldots,x_t)$ of length $t= 
\sum_{i=1}^s{\left\lceil\log_{r+1}n\right\rceil} + 1,$
where $x_i=(\alpha_1^{(i)},\alpha_2^{(i)},\ldots, \alpha_s^{(i)})$ for $i=1,2,\ldots,t.$ 

Let  $$N = \bigg\{ \sum_{l=1}^{r} l\bfx_{I_l} : I_l\subseteq [s], I_i\cap I_j = \emptyset\textrm{\  for\ }  i\neq j\bigg\}.$$ 
Now if we show that $|N| \geq n_1\cdots n_s$ then it follows that 
$D_A(\Z_n) \leq t$. Indeed, since $|G|=n_1\cdots n_s$, 
it would follow that for some distinct collections of sets $\{I_j\}_{j\in [r]} ,\{J_j\}_{j\in [r]}$ with $I_i\cap I_j = J_i\cap J_j=\emptyset$ whenever $i\ne j$, we have $$\sum_{j=1}^{r} j\bfx_{I_j}=\sum_{j=1}^{r} j\bfx_{J_j}$$ as elements in $G$. But then this yields a relation of the form $\displaystyle\sum_{j=1}^r a_j \bfx_{I_j}=0$ with $a_j\in\{\pm1,\ldots,\pm r\}$ and that is what we seek.

It is now a straightforward exercise to check that 
$|N|= (r+1)^t$. Since $t= \sum_{i=1}^s{\left\lceil\log_{r+1}n\right\rceil}+1
>\sum_{i=1}^s\log_{r+1}n,$ we have $|N| \geq n_1\cdots n_s$ and the proof is complete.\end{proof}

We quickly return to a point we made in the introduction about the finiteness of the parameter $\fD_G(k)$. By propositions \ref{HxK}, \ref{fD_ZpxZp}, \ref{pm1-pmr}, it is easy to see that when $k$ is not `too small' (and we shall prefer to be somewhat vague about what `small' means exactly, though it can easily be expounded in more precise terms), we necessarily have $\fD_G(k)<\infty$. But as we shall see, the bound in proposition \ref{pm1-pmr} is far from best possible, even in the case when $G$ is a cyclic group of prime order.

Before we conclude this section, we make one other digressive remark.  If $A\subset [1,\exp(G)-1]$ is somewhat `large', (so that $D_A(G)$ is `small') then it is probably tempting to conclude that one must have $D_B(G)$ large where $B=[1,\exp(G)-1]\setminus A$; that is however false, as the following example shows. 

Let $p$ be prime and consider $G=\Z_p, A_r=\{\pm 1,\pm 2,\cdots,\pm r\}$, and $B_r= \Z_n\setminus \{0,\pm 1,\pm 2,\cdots,\pm r\}$. We claim that $D_{B_r}(\Z_p)\leq 2$ for $r < \frac{p-1}{4}$. Then by the previous proposition, for somewhat large $r$, say $r=\Omega(p)$, we have $D_{A_r}(\Z_p)=D_{B_r}(\Z_p)=2$.

To see this,  consider the sequence $\bfx:=(1,\alpha)$ with $\alpha \neq 0$. If $\alpha \in \frac{1}{i}B_r$ for some $i$ satisfying $(r+1)\leq i\leq p-(r+1)$, then it is easy to see that there exist $a,b\in B$ such that $a.1+ b.\alpha=0$. So the only interesting case is when $\displaystyle\alpha\in \bigcap_{i=r+1}^{p-(r+1)} \frac{1}{i}A_r$. The main observation now is that if $r< \frac{p-1}{4}$, then $\bigcap_{i=r+1}^{p-(r+1)} \frac{1}{i}A_r=\emptyset$. Indeed, consider an array whose rows are indexed by the elements of $B_r$, the columns by the elements of $A_r$, and whose $(i,j)^{th}$ entry is $j/i$. An element of the intersection corresponds to picking an transversal for this array, i.e., a set of entries, one from each row, such that no two chosen elements are in the same column. But this is impossible if $p-2r-1>2r$, i.e., if $r<\frac{p-1}{4}$.

\section{Proofs of theorem \ref{fD_prime} and theorem \ref{2and4}}
In this section, we prove theorems \ref{fD_prime} and \ref{2and4}, which shall appear in the two subsections of this section. 

The main idea behind the proof of theorem \ref{fD_prime} is to consider random sets $A$. To make this more specific, suppose $0<\theta<1$. By a \textit{$\theta$-random subset of $[a,b]$}, we mean a random subset $A\subseteq [a,b]$ obtained by picking each $i\in[a,b]$ independently with probability $\theta$. Also, for a probability space we say that a sequence of events $\E_n$ occurs  \textit{with high probability} (abbreviated as \textit{whp}) if $\displaystyle\lim_{n\to\infty} \bP(\E_n)=1$. In our results, the parameter $n$ will be clear from their corresponding contexts, so we do shall not draw attention to it explicitly.

Before we state our main result more precisely, we note that  one can prove a more general upper bound for $\fD_G(k)$ for \textit{all} abelian groups. In fact, the following proposition also shows that $\fD_G(k)$ is a relevant problem only for $k\le \lceil\log_2 |G|\rceil+1$. 
\begin{prop}\label{prop1}Suppose $G$ is a finite abelian group of exponent $n$, and let $A$ be a $\theta$-random subset  of $[n-1]$, where $\theta\ge \frac{\omega(n)}{\sqrt{n}}$, where $\omega(n)\to\infty$ as $n\to\infty$. Then $D_A(G)\leq {\left\lfloor\log_{2}|G|\right\rfloor}+1$ \textit{whp}. 
\end{prop}
\begin{proof}
Suppose the set $A$ contains $x,n-x$, for some $x \in [1,n-1]$. Let $\bfy=(y_1,\ldots,y_s)$ be a $G$-sequence, and suppose $s> \log_{2}|G|.$
Consider the set  $\displaystyle\big\{\bfy_I : I\subseteq [s]\big\}$; as $I$ varies over all subsets of $[s]$ and as there are $2^s> |G|$ such summands, it follows that there exist $ J_1, J_2 \subseteq [s]$ with $J_1 \neq J_2$ such that $\bfy_{J_1} = \bfy_{J_2}.$
Set $J = (J_1\cup J_2) \setminus (J_1\cap J_2)$ and define the sequence $\bfa$ by setting $a_j = x$ for $j\in J_1\setminus J_2$ and $a_j = n-x$ if $j\in J_2\setminus J_1.$ Then  clearly,  $\langle\bfy,\bfa\rangle_J=0$, so it follows that for such $A$, we have $D_A(G)\le\log_2 |G|$.

Let $A$ be a $\theta$-random subset of $[1,n-1]$. For each $x\in [1,n-1]$, let $\E_x$ denote the event that both $x,n-x$ are in $A$, and let $\E=\bigwedge_x \overline{\E_x}$.  Since $A$ is $\theta$-random, it follows that $$\mathbb{P} (\E) = (1- \theta^2)^{\frac{n-1}{2}} \leq e^{-\theta^2 (n-1)}.$$ 
By assumption, $\theta\gg1/\sqrt{n}$, so it follows that $\bP(\E)\to 0$ as $n\to\infty$ and so we are done.\end{proof}
{\bf Remark:} A quick consequence of proposition \ref{prop1} is the following: Set $\theta=\frac{n^{\ep}}{\sqrt{n}}$. Then \textit{whp} a $\theta$-random subset of $[1,n-1]$ satisfies $D_A(G)\le\lfloor\log_2 |G|\rfloor+1$. In particular, for any $k\le\lfloor\log_2 |G|\rfloor+1$, we have $\fD_G(k)\le n^{1/2+\ep}$, for any $\ep>0$.

\subsection{Proof of theorem \ref{fD_prime}}
As mentioned earlier, the proof of theorem \ref{fD_prime} involves studying $D_A(\Z_p)$ for random $A\subset[1,p-1]$. Our main probabilistic tool here is Janson's inequality (see \cite{AlSp}, for instance).  The version of the inequality that we shall use, is given below, for the sake of completeness.

Suppose $\Omega$ is a finite set, and let $R$  a random subset of $\Omega$ where each $r \in \Omega$ is chosen independently with probability $p_r$. Let $A_i\subset\Omega$ for $i=1,2\ldots, t$, and let $\E_i$ denote the event: $A_i\subset R$.  Let $\displaystyle N=\#\{i : A_i\subset R\}, \mu:=\bE(N),\ \Delta:=\sum_{i \sim j} \bP(\E_i \wedge \E_j),$  where we write $i\sim j$ if $A_i\cap A_j\neq\emptyset$. Then one of the forms of Janson's inequality states that $\bP(N=0)\le\exp(-\mu+\frac{\Delta}{2})$, and this is what we shall use. 

The following theorem in this subsection details the nature of $D_A(\Z_p)$ when $A$ is $\theta$-random, for prime $p$:  
\begin{thm}\label{randA} Suppose $p$ is a prime and $A$ is a $\theta$-random subset of $[1,p-1]$. Let $\omega(p),\omega'(p) $ be arbitrary functions satisfying $\omega(p),\omega'(p)\to\infty$ as $p\to\infty$. Also, suppose $p$ is sufficiently large.
\begin{enumerate}
\item If $\theta>\sqrt{\frac{2\log p + \omega(p)}{p}}$, then  \textit{whp} $D_A(\Z_p)=2$.
\item If $k\ge 3$ is an integer and $\theta$ satisfies 
$$\frac{\left(3kp(\log p+\omega(p))\right)^{1/k}}{p} 
< \theta<\frac{p^{1/(k-1)}}{p\ \omega'(p)},$$ then \textit{whp} 
$D_A(\Z_p)=k$.
\end{enumerate}
\end{thm}
{\bf Remark:} One could have incorporated the first part of theorem \ref{randA} into the second more general part, but we state the theorem as we do, because the proof of the first part is simpler, and motivates and elucidates the general strategy better; the only difference comes in the finer details. 

It is clear that theorem \ref{randA} implies the result of theorem \ref{fD_prime}. Also, theorem \ref{randA} is clearly not tight as in that the theorem makes a statement only for $\frac{(p\log p)^{1/k}}{p}\ll \theta\ll\frac{p^{1/(k-1)}}{p}$. The constant $3$ in the statement of the theorem is definitely not tight (even by our method of proof), but we make no attempt to improve it to find the best possible constant in order to make the presentation more lucid.

\begin{proof} 
\begin{enumerate}
\item First observe that $D_A(\Z_n) \geq 2$ follows trivially by considering the sequence $\bfx= (1)$.  Fix a sequence $\bfx = (x_1,x_2)$ of length $2$ in $\Z_p$. Without loss of generality, we may assume that both $x_i\in\Z_p^*$.  Write  $u= x_2/x_1\neq 0$.\\

For this given sequence, consider the graph $G_u = (V_u, E_u)$, where $V_u= Z_p^*$ and for $a,b \in V_u$, $(a,b)\in  E_u$
if and only if $a=-bu$ or $b=-au$. If $A\subset[1,p-1]$ is regarded a subset  of the vertex set of $G_u$, and if $A$ is not independent in $G_u$, then by the definition of $G_u$, it follows that the sequence $(x_1,x_2)$ admits a pair $a,b\in A$ such that $ax_1+bx_2=0.$

Suppose $A$ is a $\theta$-random subset of $[1,p-1]$ and let $N_u=|\{ e\in E_u: e\subset A \}|.$ Since each vertex of $G_u$ has degree $2$, $G_u$ is a union of cycles, so it is straightforward to see that  $$\bE(N_u)=(p-1)\theta^2, \hspace{0.3cm}\varDelta:=\sum_{\substack{e\cap e'\ne\emptyset\\e\ne e'\in E_u}} \bP(e,e'\subset A)= (p-1)\theta^3.$$ By Janson's Inequality we have 
$\mathbb{P}(N_u=0)\leq e^{-(p-1)\theta^2(1-\theta/2)}$. Hence by the hypothesis on $\theta$,  
$$\bP(\textrm{There\ exists\ }u\in\Z_p^*\textrm{\ such\ that\ }N_u=0)\le \exp(-(p-1)\theta^2(1-\theta/2)+\log p)\le \exp\left(-\omega(p)\right),$$ and that completes the proof.
 

\no

\item The proof of this part is very similar to the proof of part 1, with the crucial difference being that rather than evaluate $\mu, \Delta$ precisely (which is messy), we shall use appropriate bounds in this case.  

Let $\X$ be the set of all $k$-tuples $\bfx=(x_1,\ldots,x_k)$  of elements in $\Z_p$  such that $\bfx_I\ne 0$ for all non-trivial subsets $I\subset [k]$. 

Call a $k$-tuple $\bfa=(a_1,\ldots,a_k )$ of elements in $\Z_p^*$  \textit{good} if $\langle\bfx,\bfa\rangle =0$. 
For each $i\in [k]$, let $\N_i$ denote the set of good $k$-tuples $(a_1,\ldots,a_k)$ with exactly $i$ distinct $a_j$'s, and let $n_i=|\N_i|$. We claim that for $i<k, n_i=O_k(p^{i-1})$, and that for $n_k\ge (p-1)(p-2)\cdots(p-k+1) - O_k(p^{k-2})$.

To see why, note that every good $k$-tuple in $\N_i$ corresponds to a partition $[k]=\cup_{j=0}^{i-1},I_{j}$  into $i$ non-empty, disjoint sets, and for each such partition of $[k]$, the equation $a \bfx_{I_0}+b_1 \bfx_{I_1}+\cdots+b_{i-1} \bfx_{I_{i-1}}=0$ admits a unique solution for $a$ whenever we fix choices for $b_1,\ldots,b_{i-1}\in \Z_p^*$. For $i=k$, again, for distinct choices of $b_1,\ldots,b_{k-1}$, the equation $ax_1+b_1x_2+\cdots+b_{k-1}x_k=0$ admits a unique solution for $a\in\Z_p$ and the only $k$-tuples of this sort that do not lie in $\N_k$  must correspond to the case where $a=0$ (in which case $(b_1,\ldots,b_{k-1})$ is a good $(k-1)$-tuple), or if $a=b_i$ for some $1\le i<k$, in which case these $k$-tuples are in $\N_{k-1}$, and either way, the lower bound mentioned above is satisfied. 

Let $N_{\bfx}$ denote the number of good $k$-tuples for $\bfx$ arising from the  $\theta$-random set $A$.  Then $$\mu =\bE(N_{\bfx})= \sum_{i=1}^k n_i \theta^{i}\ge \frac{1}{2}p^{k-1}\theta^k,$$ by the discussion above, for $p$ sufficiently large.

To compute $\Delta$, we set up a little additional notation. For $k$-tuples $\bfa,\bfb$, we write $\bfa\sim\bfb$ if there is some common element (not necessarily in the same position) in the sequences, and by $|\bfa\cap\bfb|$ we shall denote the number of common elements in the two $k$-tuples.  

First, observe that 
\begin{eqnarray*} 
\Delta=\sum_{\substack{\bfa\sim\bfb\\ \bfa,\bfb\textrm{\ good}}} \bP(\bfa,\bfb\subset A)&=& \sum_{i=2}^k\sum_{\substack{\bfa\sim\bfb\\ \bfa,\bfb\in\N_i}}\bP(\bfa,\bfb\subset A) \\
             &=&\sum_{i=2}^k\sum_{j=1}^{i}\#\big\{(\bfa,\bfb)\in\N_i\times\N_i: |\bfa\cap\bfb|=j\big\}\theta^{2i-j}\label{Delta}
\end{eqnarray*}
To bound this, we note that $\#\big\{(\bfa,\bfb)\in\N_i\times\N_i: |\bfa\cap\bfb|=j\big\}=O_k(p^{2i-j-2}).$ Indeed, there are $n_i=O_k(p^{i-1})$ choices for $\bfa\in\N_i$ and for a fixed $\bfa$ and for a certain fixed subset $T$ of size $j$ among the $i$ distinct elements of $\bfa$, there are at most $p^{i-j-1}$ sequences $\bfb\in\N_i$ such that $\bfb$ also contains the elements of $T$. Also, for $i=k, j=1$, the number of pairs $(\bfa,\bfb)$ is at most $k\theta^{2k-1}p^{2k-3}$. It is now a simple check to see that for $p$ sufficiently large, we have $\Delta\le 2k\theta^{2k-1}p^{2k-3}$.

Consequently, by Janson's Inequality 
$$\bP(N_{\bfx}=0) \leq e^{-\mu +\frac{\varDelta}{2}}\le \exp\left(-\frac{1}{2}\theta^kp^{k-1}+2k\theta^{2k-1}p^{2k-3}\right).$$
So, again as before, 
\begin{eqnarray*}\bP(\textrm{There\ exists\ }\bfx\textrm{\ such\ that\ }N_{\bfx}=0)&\le&\exp\left(-\frac{1}{2}\theta^kp^{k-1}+2k\theta^{2k-1}p^{2k-3}+k\log p\right)\\ &<&\exp\left(-C\omega(p)\right)\end{eqnarray*} for some constant $C>0$, by the bounds on $\theta$.

For the final part of the theorem, consider the sequence $\mathbf{1}_{k-1}:=(\underbrace{1,\ldots,1}_{k-1\textrm{\ times}})$ and let $T_{k-1}$ denote the number of $k-1$-tuples $\bfa=(a_1,\ldots,a_{k-1})\in \Z_p^{k-1}\setminus\{\bfzero_{k-1}\}$ satisfying $\langle\mathbf{1}_{k-1},\bfa\rangle=0$.  By the arguments outlined earlier, it is not hard to see that 
$$\bE (T_{k-1})\le kp^{k-2}\theta^{k-1}<\frac{k}{\omega'(p)^{k-1}}$$
by the hypothesis on $\theta$. In particular, it follows that $\bE(T_{k-1})\to 0$ as $p\to\infty$, and from this, it follows that with high probability, there exists no $\bfa\in (A\cup\{0\})^{k-1}\setminus\{\bfzero_{k-1}\}$ for which the sequence $\mathbf{1}_{k-1}$ admits an $A$-weighted zero sum subsequence. This completes the proof of theorem \ref{randA}.
\end{enumerate}\end{proof}
\subsection{Proof of theorem \ref{2and4}}
\begin{proof}
\begin{enumerate} 
\item For this part, we recall the notion of a Difference set in an abelian group (see \cite{BJL}):
\begin{defn}{\bf Difference Set:} Suppose $G$ is an abelian group of order $v$. A set $D\subset  G\setminus\{0\}$ is called a $(v,k,\lambda)$ difference set if 
\begin{enumerate}
\item $|D|=k$ and
\item for each $g\in G, g\ne 0$, there are exactly $\lambda$ pairs $(d,d')\in D\times D$ such that $d-d'=g$.\end{enumerate}
\end{defn}
If $\lambda =1$, then $D$ is called a perfect difference set.

The necessary tool  for us is a classical result due to Singer (\cite{Sin}):
\begin{thm} (Singer, \cite{Sin}) Suppose $n=q^2+q+1$ for $q$ prime, then the cyclic group  $G= \Z_n$ admits a perfect difference set of size $q+1$.\end{thm}
First by theorem \ref{bndsZp^m} we have $f^{(D)}(p,2)\geq\sqrt{p}-1$. However, a closer inspection of the same proof for the case $k=2$ reveals that the corresponding set $\X$ (in the proof of theorem \ref{bndsZp^m}) consists of all pairs $(a_1a_2)\in A^2$ itself, so we actually have a (slightly) better bound, viz., $\fD(p,2)\ge\sqrt{p-1}$, so $ f^{(D)}(p,2) \geq \left\lceil\sqrt{p-1}\right\rceil.$  

Suppose $p=q^2+q+1$ and let (by Singer's theorem)  $D\subset\Z_p^*$ be a perfect difference set of size $q+1$, and set $A=\{\theta^i : i\in D \}$, where $\theta$ is a primitive element of $\Z_p^*$. We claim that $D_A(\Z_p)=2$, so that this establishes that  $f^{(D)}(p,2) \leq \left\lceil\sqrt{p-1}\right\rceil$ and completes the proof.\\

In order to show that $D_A(\Z_p)=2$, it suffices to show that for every $u\in\Z_p^*$, the sequence $(1,u)$ admits a pair $(a_1,a_2)\in A^2$ such that $a_1+u a_2=0$. Write $-u=\theta^{i_u}$ for a unique $i_u\in [1,n^2+n]$, and since $D$ is a perfect difference set, write $i_u=d_1-d_2$ for a unique pair $(d_1,d_2)$ in $D$. Then if we set $a_i=\theta^{d_i}$ then we have $-u=a_1/a_2$, or equivalently, $a_1+ua_2=0$. 

\item For $k=2$, and all primes $p$, we now prove the more general bound $\fD(p,2)\le 2\sqrt{p}-1$. Again, we ignore the ceiling/floor notation for simplicity.  The following simple observation is key. In what follows, if $A$ is a set containing $0$ then by $A_*$ we shall mean $A\setminus\{0\}$.
\begin{obs}\label{observ} Suppose $A,B\subset\Z_p^*$ and satisfy $|A|\cdot |B|>p$. Then $\frac{B-B}{(A-A)_*}=\Z_p$. \end{obs}
To prove the observation, for each $x\in\Z_p^*$, consider the map $\phi_x:A\times B\to\Z_p^*$ given by $\phi_x(a,b):=ax+b$. Then by the assumption that $|A|\cdot |B|>p$ it follows that this map is not injective, so there exist pairs $(a,b)\neq (a',b')$ such that $\phi_x(a,b)=\phi_x(a',b')$. By the definition of $\phi_x$, this implies that $a\ne a'$, which then implies that $x=\frac{b'-b}{a-a'}\in\frac{B-B}{(A-A)_*}$.

Consider the set $A=[-\sqrt{p},\sqrt{p}]_*$. Then note that $A=(B-B)_*$ where $B=[1,\sqrt{p}+1]$. Since $|B|^2>p$, by the observation above, we have $\frac{A}{A}=\frac{(B-B)_*}{(B-B)_*}=\Z_p^*$. This implies that for every $u\in\Z_p^*, -u\in \frac{A}{A}$. Since $|A|\le 2\sqrt{p}-1$, that completes the proof.

\item Before we start with the proof of the 3rd part of this theorem, we shall state a reformulation of what we seek:  For any $k\ge 1$, to find an upper bound for  $\fD(p,2k)$, it suffices to construct a set $A\subset\Z_p^*$ of an appropriate size such that for any $\alpha_1,\ldots,\alpha_{k-1},\beta_1,\ldots,\beta_{k-1}\in\Z_p^*$  
\begin{eqnarray*}\Z_p^*\subset \frac{A+\alpha_1 A+\cdots+\alpha_{k-1}A}{A+\beta_1 A+\cdots+\beta_{k-1}A}.\end{eqnarray*}
To see why, note that $D_A(\Z_p)=k$ implies that for any $x_i\in\Z_p^*, 1\le i\le k,$ we have $0\in Ax_1+\cdots+Ax_{2k}$, or equivalently, 
$$-\frac{x_1}{x_{k+1}} = \frac{a_{k+1}+a_{k+2}(x_{k+2}/x_{k+1})+\cdots+a_{2k}(x_{2k}/x_{k+1})}{a_1+a_2(x_2/x_1)+\cdots+a_k(x_k/x_1)},$$ and if the aforementioned condition holds, then this is indeed satisfied. 
So in order to show that $\fD(p,4)\le O(p^{1/4})$, it suffices to construct a set $A$ with $|A|\le Cp^{1/4}$ for some constant $C>0$ such that 
$$\mathbb{F}_p^* \subset \frac{A+\alpha A}{A+\beta A}$$
for all $\alpha, \beta \in \mathbb{F}_p^*$.

For an integer $t$, let $X_t:=(t[-L,L])=\{-Lt,\ldots,-t, 0, t,\ldots,Lt\}$ with $L=C_0p^{1/4}$, for some large constant $C_0$. The following properties of $X_t$ are evident:
\begin{enumerate}
\item[i.] $\alpha X_t=X_{\alpha t}$. 
\item[ii.] For $s\ne t$, $X_s+X_t$ contains a subset $Y_{s,t}$ of size at least $|X_s+X_t|/4$ such that $Y_{s,t}-Y_{s,t}\subset X_s+X_t$. This follows from the simple fact that this is a \textit{generalized arithmetic progression (GAP)} of rank $2$, and this is a general property of GAPs. (see for instance \cite{TaoVu}, chapter 2. However, this property is an easily verified thing, and does not need any specialized tools).
\end{enumerate}

Our set $A$ will be of the form $A=\displaystyle\left(\bigcup_{t\in J}X_t\right)_*$ for some set $J$ with $1\in J$ and $|J|=O(1)$. We shall notate $X_1$ by $I$ for convenience. The bound on $|J|$ also will be specified  later as a function of $C_0$. 

What we shall do more specifically is to exhibit such an $A$ (for a suitable set $J$) that satisfies the following property: For each $\alpha\in\Z_p^*$, the set $A+\alpha A$ contains a subset of the form $I+X_t$ of size greater than $4p^{1/2}$. Then by property (ii) listed above, there exists $\tilde{Y_{t}}\subset I+X_t$ with $|\tilde{Y_{t}}|>p^{1/2}$ such that $\tilde{Y_{t}}-\tilde{Y_{t}}\subset I+X_t\subset A+\alpha A$.  Since this holds for each $\alpha\in\Z_p^*$, observation \ref{observ} implies that $\frac{A+\alpha A}{A+\beta A}\supset\Z_p^*$.

So, what we need to do is to exhibit a set $J$, and $A=\displaystyle\bigcup_{t\in J}X_t \setminus\{0\}$ such that the aforementioned property is satisfied.

At this juncture, we need a lemma and in order to state that we shall introduce some further terminology. For integers $\xi,\eta\in [1,p-1]$, we shall regard the sets $[-\xi,\xi]=\{-\xi,-\xi+1,\ldots,-1,0,1,\ldots\xi-1,\xi\}$ and $[1,\eta]=\{1,2,\ldots,\eta\}$ as subsets of $\Z_p$. Define
 $$S(\xi,\eta):=\frac{[-\xi,\xi]_*}{[1,\eta]}.$$
We need introduce one further bit of terminology. For a given $t\in\Z_p^*$, we say that $\alpha\in\Z_p^*$ is good for $t$ if $|(I+\alpha X_t)_*|\ge \frac{L^2}{400}$. As before, $I=X_1=[-L,L]$.
\begin{lem}\label{small} Suppose $\alpha\in\Z_p^*$ is not good for $t$, then $\alpha\in t^{-1}S(2L,\frac{L}{100})$.
\end{lem} 
\begin{proof} ({\bf of the lemma}): 
The proof of the lemma is structured as follows: The set $I+X_t$ 
can be viewed as the union of intervals $2L$, each centered at an element of the form $it$ for $1\le i\le L$. As $i$ varies, the intervals $[it-L,it+L]\subset \Z_p$ appear as intervals `winding around' $\Z_p$ (viewed cyclically). If there are sufficiently many among these that are pairwise disjoint, then their union has a large size. So, if the union does not have large size, then (as we shall show) there is a relatively small multiple of $t$ that lands us very close to zero in $\Z_p$. 

Let us now get to the details. Let $M\in\bN$ be such that  $2Mt<p<(2M+1)t$.  Let $T_0=\emptyset$. For $i\ge 1$, by the 
$i^{th}$ iteration round set we shall refer to the set $T_i$ formed by taking the union of the corresponding set $T_{i-1}$ that 
arises iteratively from the first $i-1$ iterations, and intervals of the `next collection'  of intervals of the form $[jt-L,jt+L]$ 
to the set build thus far till we `go around' $\Z_p$ again. We shall consider this iteration process over $L/200$ rounds. We say 
that the $i^{th}$ iteration is valid if the intervals that are  added to $T_{i-1}$ to form $T_i$ are all pairwise disjoint and 
also disjoint to $T_{i-1}$. In particular, $T_i$ contains the disjoint union of at least $2Mi$ intervals of length $2L$ each.

We now make a few remarks that will outline some 
assumptions/observations that we shall make without loss of generality in the remainder of the proof. 
\begin{enumerate}
\item[a.] Since we shall work with $p$ large, and our definitions concern lower bounds for $|A_*|$ for some appropriate set $A$, we shall instead work with corresponding lower bounds for $|A|$ itself. This makes no significant difference.
\item[b.] Since $\alpha X_{t}=X_{\alpha t}$, it will suffice if we show the following: If $|I+X_t|<\frac{L^2}{200}$ then $t\in S(2L, L/100)_*$.
\item[c.]  Since $I+X_t=\displaystyle\bigcup_{i=-L}^{L} [it-L,it+L]$, each iteration round (by the definition of $M$) will consist of roughly $2M$ intervals. In our analysis of the $i^{th}$ iteration round,  we shall consider the addition of the intervals centered at the elements of the form $2Mit, (2Mi+1)t,\ldots,2M(i+1)t$ instead rather of those centered at the elements $Mit, -Mit, (Mi+1)t,-(Mi+1)t,\ldots,M(i+1)t,-M(i+1)t$. This is because our analysis only depends on the nature of the modular arithmetic in $\Z_p$. Consequently, this makes no difference to the final conclusion, since it simply is equivalent to `centering'  our attention at a different element of $\Z_p$ instead of $0$.
\item[d.] If $t\le 2L$, then clearly, $t\in S(2L,L/100)$, so we may assume that $t>2L$. 
\end{enumerate}

Suppose first  $2Mt+L\ge p-L$; then the first iterate set itself does not admit $2M$ pairwise disjoint intervals. However, the assumption that $t>2L$ implies that the first $2M-1$ intervals are indeed pairwise disjoint, so $|I+X_t|\ge (2M-1)(2L)$, and if $(2M-1)\ge L/400$, then we have  $|I+X_t|\ge L^2/200$ contrary to the assumption. Hence we have $2M<L/400$. Since in $\Z_p$ we have $2Mt\in[-2L,-1]$ our bound on $M$ implies that $t\in S(2L, L/100)$. Now suppose that $2Mt+L< p-L$ and $(2M+1)t-L\le p+L$. In other words, the set $T_1$ is a disjoint union of $2M$ intervals, but the first interval of the $2^{nd}$ iterate set, namely, the interval $[(2M+1)t-L,(2M+1)t+L]$ is not disjoint from the first interval of the first iterate set. Since the first $2M$ intervals are pairwise disjoint, we must have $|I+X_t|\ge (2M)(2L)$. Again,  by the same argument as above, and the fact that $(2M+1)t\in[1,2L]$ in $\Z_p$ implies that $t\in S(2L, L/100)$.
 
The reasoning of the preceding discussion can be extended over several iterates as well. Suppose the first $r-1$ iterations are valid but the $r^{th}$ iteration is not valid; if $2M(r-1)(2L) \ge L^2/200$, then  $|T_{r-1}|\ge L^2/200$ contrary to the hypothesis, so we must have $2M<\frac{L}{400(r-1)}$. Let $0\le s\le r-1$ be such that $(2Mr+s)t<pr<(2Mr+s+1)t$. The following observation is key: Since the $r^{th}$ iteration is not valid, \textit{it is necessarily the case that}  some interval of the form $[(2Mr+r')t-L, (2Mr+r')t+L]$, with $r'\le s$,  intersects some other interval $[(2M\ell +k)t-L,b(2M\ell +k)t+L]$, with $0\le\ell<r$. In particular, it follows that 
$$(2M(r-\ell) +(r'-k))t\in [-2L,2L] \textrm{\ in\ }\Z_p,$$ and by the bound on $2M$ from before, it follows that if $r-1\le L/200$, we have $t\in S(2L, L/100)$. If however, $L/200$ iterations are all valid, then again, we must have 
$$|I+X_t|\ge |T_{L/200}|\ge (2M)\cdot(L/200)\cdot(2L)>L^2/200$$ 
and that contradicts the assumption that $\alpha\in\Z_p^*$ is not good for $t$; hence the iteration rounds cannot all be valid for $L/200$ rounds, and that completes the proof.
\end{proof}

For $t,s\in[1,(p-1)/2]$, since $|X_s+\alpha X_t|=|s(I+\alpha X_{(t/s)})|$ it follows that 
$$|X_s+\alpha X_t|\geq \frac{L^2}{400}\textrm{\ if\ and\ only\ if\ } \alpha\textrm{\  is\ good\ for\ } (t/s).$$


For simplicity, we shall now denote by $S$ the set $S(2L,\frac{L}{100})$. Consider the hypergraph $\Hy$ whose vertex set is $\Z_p^*$ and whose edge set consists of all dilates of $S$, i.e., the edge set of $\Hy= \{xS : x\in \Z_p^*\}$. 

We claim that there exists a positive integer $N=O(1)$ such that  $\Hy$ is not $N$-intersecting, i.e., there exist  $x_1,\ldots, x_N \in \Z_p^*$
such that $$x_1S\cap \cdots \cap x_NS =\emptyset.$$
If the claim holds, then consider the set $$A=\displaystyle\bigcup_{i=1}^{N}X_{x_i^{-1}}\cup I.$$ This set will be our desired $A$. 

Indeed, suppose $\alpha\in\Z_p^*$. Since $\displaystyle\bigcap_0^{N} x_iS=\emptyset$, $\alpha\not\in x_i S$ for some $i$. This in turn (by lemma \ref{small}) implies that  $\alpha$ is good for  $x_i^{-1}$, or equivalently, $|I+\alpha X_{x_i^{-1}}|\ge\frac{L^2}{400}>p^{1/2}$ and since $A+\alpha A\supset I+\alpha X_{x_i^{-1}}$, we are through. So, the proof of the theorem will be complete if the aforementioned claim has been proven.

Note first, that \begin{eqnarray}\bigg|\frac{S}{S}\bigg|\leq |S|^2= L^2/25=\frac{C_0^4}{25}p. \textrm{ \ Consequently, \ } \frac{1}{p}\bigg|\frac{S}{S}\bigg|\leq \frac{C_0^4}{25}. \end{eqnarray}

Suppose $x$ is chosen uniformly at random from $\Z_p^*.$ Let 
$$N(x)= |\{(s_1,s_2): s_i\in S \text{ satisfying } x=\frac{s_1}{s_2}\}|.$$ 
Then
$$\mathbb{E}(N(x)) = \sum_{(s_1,s_2)\in S^2} \mathbb{P}\left(x=\frac{s_1}{s_2}\right) =\frac{|S|^2}{p-1}\leq \frac{C_0^4p}{25(p-1)}< \frac{C_0^4}{12}$$
so, by the Markov Inequality, $$\mathbb{P}(N(x)> C_0^4/6)\leq \frac{1}{2}.$$ 

Let NORMAL := $\{x \in \Z_p^* : N(x) \leq C_0^4/6\}.$ By the probability estimate above,  
$$\mathbb{E}(|\text{NORMAL}|)> \frac{p-1}{2}.$$

The relevant observation regarding  $x\in$ NORMAL is that for such 
$x$, $|S\cap xS| \leq 2C_0^4.$ Indeed, suppose $x\in$ NORMAL . 
If $S\cap xS=\emptyset$ then there is nothing to prove. Let $S\cap xS = \{y_1,\ldots, y_{k}\}$ for some $k$.  Then $y_i= xs_i=s_i'$ for some  $s_i,s_i' \in S$ where $1\leq i\leq k$, and the $s_i$ are all distinct. Hence $x = \frac{s_i'}{s_i}$ for $1\leq i\leq k$, which implies that $x$ admits at least $k$ such expressions as the ratio of two elements of $S$. Since $x\in $ NORMAL, it follows that  $k\leq C_0^4/6.$

Set $N= C_0^4/3$ and pick $x_1 \in$ NORMAL. By the preceding discussion, $|S \cap x_1S|\le C_0^4/6.$ Write $S \cap x_1S=\{a_1,\ldots,a_k\}$. Now pick $x_2\neq x_1$ such that $a_1 \notin x_2S$, and continuing this process, (after having picked $x_1,\ldots,x_{i-1}$), pick $x_i\in\textrm{NORMAL}\setminus\{x_1,\ldots, x_{i-1}\}$ such that $a_{i-1} \notin x_iS.$ These choices are all possible since $|\textrm{NORMAL}|\ge(p-1)/2$, and the number of forbidden choices (at each step of this process) is at most $(C_0^4/3)|S|\le Cp^{1/2}$ for some fixed constant $C>0$, so for large enough $p$, there is always room for choosing such an $x_i$. 

But then we must have $S\bigcap x_1S\bigcap \cdots \bigcap x_NS =\emptyset$ since any element $x$ in the intersection must be $a_i$ for some $i$, but by the choices of the $x_i$, we have $a_i\not\in x_{i+1}S$, and that is a contradiction. This completes the proof of the claim, and the theorem as well. 
\end{enumerate}
 \end{proof}
 {\bf Remark:} We have not made any attempts to even describe the constant $C_0$ explicitly in the proof of the last part of the theorem above,; one could, if one were so inclined, determine a concrete value of $C_0$ for which the theorem works, but we believe that to be somewhat futile since we believe that in reality $\fD(p,4)\le (1+\ep))p^{1/4}$ (for all $\ep>0$; please see the first remark in the next section) this method may not take us anywhere close to the best possible bound.

\section{Concluding Remarks}
\begin{itemize}
\item As we have stated earlier, we strongly believe that in fact $\fD(p,k)=\Theta(p^{1/k})$ for all sufficiently large $p$. But in fact, we are also inclined to believe that in fact $\fD(p,k)\le (1+o(1))p^{1/k}$ though we can prove neither statement now. The best upper bound for $\fD(p,2)$ that we can prove (for all prime $p$) is $(2/\sqrt{3})\sqrt{p}=1.154\ldots\sqrt{p}$. This follows from some results on the existence of differences bases for $[n]$, (see \cite{BanGav}). 
\item One may frame the problem of obtaining an upper bound for $\fD(p,2k-1)$ (in an analogous manner to that in the proof of theorem \ref{2and4}) by constructing a set $A$ such that for any $\alpha_1\ldots,\alpha_k, \beta_1,\ldots,\beta_{k-1}\in\Z_p^*$ such that $$\Z_p^*\subset \frac{A+\alpha_1 A+\cdots+\alpha_k A}{A+\beta_1 A+\cdots+\beta_{k-1}A}.$$ So, for instance, to prove that $\fD(p,3)\le O(p^{1/3})$ amounts to constructing a set of the appropriate size such that $\Z_p^*\subset \frac{A+\alpha A}{A}$. But this asymmetry in the framing makes the problem of $\fD(p,2k)$ easier to approach in this manner.
\item One very natural counterpart to the problem that is the focus of this paper is the corresponding dual problem: For a given finite group $G$, determine 
$$\max\{D_A(G) : |A|=k, A\subset[1,\exp(G)-1]\}.$$ 
For instance, it is known that $D_A(\Z_p)=\lceil p/k\rceil$ if 
$A=\{1,\ldots,k\}$ for  $1\le k\le p-1$ (see \cite{AR},\cite{Adh et al}), so this corresponding maximum is at least $\lceil p/k\rceil$. It turns out, that for $p$ prime, one can show that this maximum is at most $\lceil p/k\rceil$ as follows (this result also appears in \cite{Adetal}, with a different proof): \\
Suppose $A$ is a set of size  $\lceil p/k\rceil$. We shall show that any sequence $\bfx$ of length $\lceil p/k\rceil$, there exist $\bfa\in A\cup \{0\})^k \setminus\{\bfzero_k\}$ such that $\langle \bfx,\bfa\rangle=0$. Write $p = (m-1)k+r$ for $0<r<k$, so that $m= \left\lceil p/k\right\rceil.$ Let $\X:=(x_1\ldots x_m)$ be a sequence of non-zero elements of $\Z_p$, and let $S=A\cup\{0\}$. Consider the polynomial $g(X_1,\cdots,X_m)= \left((\sum_{i=1}^{m}x_iX_i)^{p-1}-1\right)X_1^{k+1-r}$, and let $g(X_1,\cdots,X_m)= \left((\sum_{i=1}^{m}x_iX_i)^{p-1}-1\right)X_1^{k+1-r}$. The coefficient of $\prod_{i=1}^{m}X_i^{k}$ in $g$ equals $\binom{p-1}{r-1,k,\cdots,k}x_1^{r-1}\prod_{i=2}^mx_i^k\neq 0,$ since $x_i\ne 0$, so by the Combinatorial Nullstellensatz (see \cite{Alon}) it follows that there is a choice of $a_i\in S_{i}$ for each $1\le i\le m$ with $a_1\ne 0$ (since $k+1-r>0$) such that $\sum_i a_ix_i=0$. In fact this proof also works even when we assign arbitrary lists of size $\lceil p/k\rceil$ for each non-zero element of $\Z_p$ and we are only allowed to pick coefficients from the corresponding list of each element, so that a corresponding list-weighted version of the Davenport constant also admits the same upper bound. 

The same ideas can be extended to show that if $n=p_1\cdots p_r$ is square-free, and $A$ is a subset of $[1,n-1]$ of size $k$ such that the set $A\pmod{p_i}:=\{a\pmod{p_i}:a\in A\}$ also has size $k$, then any sequence $(x_1\ldots,x_m)$ in $\Z_n$ with $m\ge \frac{\left\lceil p_i/k\right\rceil n}{p_i}$ admits an $A$-weighted zero-sum subsequence. Indeed, suppose $A_i:=A\pmod{p_i}$ has size $k$. Since $n$ is square-free, $\Z_n \cong \Z_{p_i}\times \Z_{n/p_i}.$ Write $p_i= k\lambda_i+ r_i,$ for some $\lambda_i, r_i$ where $0 \leq r_i < k$, so that $\left\lceil{p_i/k}\right\rceil = \lambda_i + 1.$ Let $\X=(x_1,\ldots, x_{m_i})$ be a sequence of elements in $\Z_n$, where $m_i=\frac{n\left\lceil p_i/k\right\rceil}{p_i}.$
Write $x_j= (y_j, z_j)\in  \Z_{p_i}\times \Z_{n/p_i}$  for each $j=1,\ldots,m_i$; similarly, write $a=(a',a'')$ where $a'\in\Z_{p_i}$ and $a''\in\Z_{n/p_i}$. Regroup the sequence $\X$ into $n/p_i$ segments of length $\lceil p_i/k\rceil$ each.  It follows that for each $1\le j\le n/p_i$ and $1\le\ell\le \lceil p_i/k\rceil$, there exist $a'_{\ell, j}\in A_i\cup\{0\}$, not all zero, such that $$\sum_{\ell=(j-1)\lceil p_i/k\rceil+1}^{j\lceil p_i/k\rceil} a'_{\ell,j}y_{\ell}=0$$ in $\Z_{p_i}$.
Writing $t=\lceil p_i/k\rceil$ for convenience, we note that in particular, we have the sequence (from our regrouping)
$$\Y= \left(\left(0,\sum_{\ell=1}^{t}a''_{\ell,1}z_{\ell}\right),\left(0,\sum_{\ell=t+1}^{2t} a''_{\ell,2}z_{\ell}\right),\ldots,\left(0,\sum_{j=(n/p_i-1)t+1}^{nt/p_i}a''_{\ell,(n/p_i-1)}z_{\ell}\right) \right)$$ in $\Z_{p_i}\times \Z_{n/p_i}$ of length $n/p_i$. Note that the first coordinates are all zero by our choices of $a\in A_i\cup\{0\}$. But since $D(\Z_m)=m$, every sequence of length $n/p_i$ in $\Z_{n/p_i}$ admits a non-trivial zero-sum subsequence,  so we are through. An immediate consequence of this is the following: 

For $N:=\max\left\{\left\lceil \frac{p_i}{\sqrt{k}}\right\rceil\frac{n}{p_i}:1\le i\le r\right\},$ any set $A\subset[1,n-1]$ of size $k$, and any $\Z_n$-sequence $\bfx=(x_1,\ldots,x_N)$ of length $N$, there exists $\bfa\in(A\cup \{0\})^k \setminus\{\bfzero_k\}$ such that $\langle \bfx,\bfa\rangle=0$, so in particular, 
$$\max\{D_A(\Z_n):|A|=k\}\le \max\left\{\left\lceil \frac{p_i}{\sqrt{k}}\right\rceil\frac{n}{p_i}:1\le i\le r\right\}. $$

\end{itemize}

\section*{Acknowledgments}
The second author's research partially supported by NSFC with grant no. 11681217.

\end{document}